\documentclass[12pt]{article}%
\usepackage{amssymb}
\usepackage{amsfonts}
\usepackage{amsmath}%
\usepackage{graphicx}
\providecommand{\U}[1]{\protect \rule{.1in}{.1in}}
\newtheorem{theorem}{Theorem}

\newtheorem{lemma}[theorem]{Lemma}

\newtheorem{remark}[theorem]{Remark}
\newenvironment{proof}[1][Proof]{\noindent \textbf{#1.} }{\  \rule{0.5em}{0.5em}}
\begin{document}

\date{}
\title{The Tracy-Widom distribution is not infinitely divisible.}
\author{J. Armando Dom\'{\i}nguez-Molina\\Facultad de Ciencias F\'{\i}sico-Matem\'{a}ticas\\Universidad Aut\'{o}noma de Sinaloa, M\'{e}xico}
\maketitle

\begin{abstract}
The classical infinite divisibility of distributions related to eigenvalues of
some random matrix ensembles is investigated. It is proved that the $\beta
$-Tracy-Widom distribution, which is the limiting distribution of the largest
eigenvalue of a $\beta$-Hermite ensemble, is not infinitely divisible.
Furthermore, for each fixed $N\geq2$ it is proved that the largest eigenvalue
of a GOE/GUE random matrix is not infinitely divisible.

\end{abstract}


\section{Introduction}

Random matrix theory is an important field in probability, statistics and
physics. One of the aims of random matrix theory is to derive limiting laws
for the eigenvalues of ensembles of large random matrices. In this sense this
note will focus in the study the behavior of eigenvalues of two types of
matrix ensembles, the invariant Hermite and the tridiagonal $\beta$-Hermite.

The invariant Hermite ensembles consist of the Gaussian orthogonal, unitary,
or symplectic ensembles, G(O/U/S)E, which are ensembles of $N\times N$ real
symmetric, complex Hermitian or Hermitian real quaternion matrices, $H,$
respectively, whose matrix elements are independently distributed random
Gaussian variables with probability density function (PDF) proportional,
modulo symmetries, to
\[
\exp \left(  -\tfrac{\beta}{4}\mathrm{tr}H^{2}\right)  ,
\]
here, $\beta=1,2$ or $4$ is used for the G(O/U/S)E ensembles, respectively.
The joint PDF of their ordered eigenvalues $\lambda_{1}\leq \cdots \leq
\lambda_{N}$ is given by%
\begin{equation}
k_{N,\beta}\prod_{1\leq i<j\leq N}\left \vert \lambda_{i}-\lambda
_{j}\right \vert ^{\beta}\exp \left(  -\tfrac{\beta}{4}\sum_{i=1}^{N}\lambda
_{i}^{2}\right)  ,\label{pdfEigen}%
\end{equation}
where $k_{N,\beta}$ is a non negative constant and for $\beta=1,2$ or $4,$ it
can be computed by Selberg's Integral formula (see \cite[Theorem 2.5.8]%
{agz10}). The PDF (\ref{pdfEigen}) exhibits strong dependence of the
eigenvalues of the G(O/U/S)E ensembles. For more details related to these
ensembles see \cite{m04}, \cite{dg09}, \cite{f10}, \cite[sections 2.5 and 4.1
]{agz10}. The law (\ref{pdfEigen}) has a physical sense since it describes a
one-dimensional Coulomb gas at inverse temperature $\beta$, \cite[Section
1.4]{f10}.\bigskip

Each member of the G(O/U/S)E ensembles leads, by the Householder reduction, to
a symmetric tridiagonal matrix $(H_{N}^{\beta})_{N\geq1}$ of the form
\begin{equation}
H_{N}^{\beta}:=\frac{1}{\sqrt{\beta}}\left[
\begin{array}
[c]{ccccc}%
N\left(  0,2\right)  & \chi_{\left(  n-1\right)  \beta} &  &  & \\
\chi_{\left(  n-1\right)  \beta} & N\left(  0,2\right)  & \chi_{\left(
n-2\right)  \beta} &  & \\
& \ddots & \ddots & \ddots & \\
&  & \chi_{2\beta} & N\left(  0,2\right)  & \chi_{\beta}\\
&  &  & \chi_{\beta} & N\left(  0,2\right)
\end{array}
\right]  ,\label{Tridi}%
\end{equation}
where $\chi_{t}$ is the $\chi$-distribution with $t$ degrees of freedom, whose
probability density function is given by $f_{t}\left(  x\right)
=2^{1-t/2}x^{t-1}e^{-x^{2}/2}/\Gamma \left(  t/2\right)  .$ Here,
$\Gamma \left(  \alpha \right)  =\int_{0}^{\infty}v^{\alpha-1}e^{-v}dv$ is
Euler's Gamma function. The matrix (\ref{Tridi}) has the important
characteristic that all entries in the upper triangular part are independent.
Trotter \cite{t84} apply the Householder reduction for the symmetric case
($\beta=1)$ while for the unitary and symplectic cases ($\beta=2,4$) the
former reduction was applied by Dumitriu and Edelman \cite{de02}. Furthermore
the late considers the ensemble (\ref{Tridi}) for general $\beta>0$ proving
that in this case the PDF of the ordered eigenvalues of $H_{N}^{\beta}$ is
still the PDF (\ref{pdfEigen}), see \cite[Chapter 20]{abd11} and \cite[Section
4.1]{agz10}. This matrix model it will named $\beta$-Hermite ensemble.\bigskip

The classical Tracy-Widom distribution is defined as the limit distribution of
the largest eigenvalue of a G(O/U/S)E random matrix ensemble. It is important
due to its applications in probability, combinatorics, multivariate
statistics, physics, among other applications. Tracy and Widom \cite{tw02},
\cite{tw09} have written concise reviews for the situations where their
distribution appears.

The $\beta$-Tracy-Widom distribution is defined as the limiting distribution
of the largest eigenvalue of a $\beta$-Hermite ensemble. In the case
$\beta=1,2$, $4$ the $\beta$-Tracy-Widom distribution coincides with the
classical Tracy-Widom, \cite{rrv11}.\bigskip

The main purpose of this paper is to determine the infinite divisibility of
the classical Tracy-Widom and $\beta$-Tracy-Widom distributions as well as the
infinite divisibility of the largest eigenvalue of the finite dimensional
random matrix of GOE and GUE ensembles.\bigskip

Recall that a random variable $X$ is said to be \textit{infinitely divisible}
if for each $n\geq1$, there exist independent random variables $X_{1}%
,...,X_{n}$ identically distributed such that $X$ is equal in distribution to
$X_{1}+\cdots+X_{n}$. This is an important property from the theoretical and
applied point of view, since for any infinitely divisible distribution there
is an associated L\'{e}vy process; Sato \cite{s13}, Rocha-Arteaga and Sato
\cite{ras03}. These jump processes have been recently used for modelling
purposes in a broad variety of different fields, including finance, insurance,
physics among others; see Barndorff-Nielsen \emph{et al.} \cite{bn01}, Cont
and Tankov \cite{ct03} and Podolskij \emph{et al. }\cite{pstv16} and for a
physicists point of view see Paul and Baschnagel \cite{pb13}. Other
applications concern deconvolution problems in mathematical physics, Carasso
\cite{c92}.\bigskip

This note is structured as follows: in section 2 it is presented preliminary
results on the tail behavior of the classical and generalized $\beta
$-Tracy-Widom distribution, useful to analyze infinite divisibility. The non
infinite divisibility of the classical and generalized $\beta$-Tracy-Widom
distribution is proved in Section 3. In Section 4 it is shown that for each
$N\geq2$ the largest eigenvalue of a G(O/U)E ensemble is not infinitely
divisible. Finally in Section 5 it is presented more new results and some open problems.

\section{Tracy-Widom distributions}

\subsection{Classical Tracy-Widom distribution}

It is well known that the unique possible limit distributions for the maximum
of independent random variables are the Gumbel, Fr\'{e}chet and Weibull
distributions. To classify the limit laws for the maximum of a large number of
non independent random variables is still open problem. A possible strategy is
to deal with particular models of non independent random variables.\bigskip

The eigenvalues of random matrices provide a good example of such non
independent random variables. For the Gaussian ensembles, i.e. for $N\times N
$ matrices with independent Gaussian entries, the joint density function of
their eigenvalues, $\lambda_{1}\leq \cdots \leq \lambda_{N}$ is given by
(\ref{pdfEigen}), and because the Vandermonde determinant $\prod_{1\leq
i<j\leq N}\left \vert \lambda_{i}-\lambda_{j}\right \vert ^{\beta}$ they are
strongly dependent. Due to the non independence of random variables with
density function (\ref{pdfEigen}) it follows that the limit distribution of
$\lambda_{\max}=\lambda_{N}$ is not an usual extreme distributions. The
distribution of $\lambda_{\max}$ converges in the limit $N\rightarrow \infty$
to the Tracy--Widom laws.\bigskip

Tracy-Widom, \cite{tw94}, \cite{tw96} proved that the following limit, which
is denoted by $F_{\beta},$ exists%
\[
F_{\beta}\left(  x\right)  :=\lim_{N\rightarrow \infty}P\left[  N^{1/6}\left(
\lambda_{\max}-2\sqrt{N}\right)  \leq x\right]  ,\beta=1,2,4
\]
and in this case%
\[
F_{2}\left(  x\right)  =\exp \left(  -\int_{x}^{\infty}\left(  s-x\right)
\left[  q\left(  s\right)  \right]  ^{2}ds\right)  ,
\]
where $q$ is given in terms of the solution to a Painlev\'{e} type II
equation, and%
\[
F_{1}\left(  x\right)  =\exp \left[  E\left(  x\right)  \right]  F_{2}%
^{1/2}\left(  x\right)  ,\  \ F_{4}\left(  x\right)  =\cosh \left[  E\left(
x\right)  \right]  F_{2}^{1/2}\left(  x\right)  ,
\]
where $E\left(  x\right)  =-\tfrac{1}{2}\int_{x}^{\infty}q\left(  s\right)
ds.$

Furthermore it can be deduced from \cite{tw94}, \cite{tw96}, \cite{tw09} and
\cite{bbd07} (see also \cite[Exercise 3.9.36]{agz10}) that the asymptotics for
$F_{\beta}\left(  x\right)  $ as $x\rightarrow \infty,$ for $\beta=1,2$ or $4$
is,
\begin{equation}
F_{\beta}\left(  -x\right)  =\exp \left \{  -\frac{1}{24}\beta x^{3}\left[
1+o\left(  1\right)  \right]  \right \}  ,\  \label{ltail}%
\end{equation}%
\begin{equation}
1-F_{\beta}\left(  x\right)  =\exp \left \{  -\frac{2}{3}\beta x^{3/2}\left[
1+o\left(  1\right)  \right]  \right \}  ,\label{rtail}%
\end{equation}
where $o\left(  1\right)  $ is the little-o of 1 which means that
$\lim \limits_{x\rightarrow \infty}o\left(  1\right)  =0.$\bigskip

With the tail probabilities (\ref{ltail}) and (\ref{rtail}) it is possible to
conclude that the Tracy-Widom distribution is not infinitely divisible for
$\beta=1,2,4$ using the results of Section 4. Nevertheless it is convenient go
to next section from which it will arrive at the non infinite divisibility of
$\beta$-Tracy-Widom distribution for any $\beta>0,$ by considering the limit
distribution of the largest eigenvalue of a matrix $H_{N}^{\beta}$ defined in
(\ref{Tridi}) for any $\beta>0$.

\subsection{$\beta$-Tracy-Widom distribution}

The tridiagonal $\beta$-Hermite ensemble (\ref{Tridi}) can be considered as a
discrete random Schr\"{o}dinger operator. This stochastic operator approach to
random matrix theory was conjectured by Edelman and Sutton \cite{es07}, and
was proved by Ram\'{\i}rez, Rider and Vir\'{a}g \cite{rrv11}, who in
particular established convergence of the largest eigenvalue of a $\beta
$-Hermite ensemble for any $\beta>0$. Let $\lambda_{\max}=\lambda_{N}\left(
H_{N}^{\beta}\right)  ,$ with $H_{N}^{\beta}$ defined as in (\ref{Tridi}), in
\cite{rrv11} it is shown the existence of a $\beta$-Tracy-Widom random
variable $TW_{\beta}$ such that%
\[
N^{1/6}\left(  \lambda_{\max}-2\sqrt{N}\right)  \overset{d}{\underset
{N\rightarrow \infty}{\longrightarrow}}TW_{\beta},
\]
where the $\beta$-Tracy-Widom random variable is identified through a random
variational principle:%
\[
TW_{\beta}:=\sup_{f\in L}\left \{  \tfrac{2}{\beta}\int_{0}^{\infty}%
f^{2}\left(  x\right)  db\left(  x\right)  -\int_{0}^{\infty}\left[  \left(
f^{\prime}\left(  x\right)  \right)  ^{2}+xf^{2}\left(  x\right)  \right]
dx\right \}  ,
\]
in which $x\rightarrow b\left(  x\right)  $ is a standard Brownian Motion and
$L,$ is the space of functions $f$ which satisfy $f\left(  0\right)
=0,\  \int_{0}^{\infty}f^{2}\left(  x\right)  dx=1,\int_{0}^{\infty}\left[
\left(  f^{\prime}\left(  x\right)  \right)  ^{2}+xf^{2}\left(  x\right)
\right]  dx<\infty.$\bigskip

The cases $\beta=1,2,4$ coincide with the classical Tracy-Widom distribution
$F_{\beta}\left(  x\right)  =P\left(  TW_{\beta}\leq x\right)  ,$
\cite{rrv11}. Ram\'{\i}rez, \emph{et al.} \cite{rrv11} also proved that the
tails of $TW_{\beta}$ are given by (\ref{ltail}) and (\ref{rtail}) for any
$\beta>0.$

\section{Non infinite divisibility of Tracy-Widom distributions}

First recall a well known result on a characterization of the Gaussian
distribution in terms of the tail behavior; see \cite[Corollary 4.9.9]{svh03}:
a non-degenerate infinitely divisible random variable $X$ has a normal
distribution if, and only if, it satisfies%
\begin{equation}%
%TCIMACRO{\QATOP{\lim\sup}{x\rightarrow\infty}}%
%BeginExpansion
\genfrac{}{}{0pt}{}{\lim \sup}{x\rightarrow \infty}%
%EndExpansion
\frac{-\log P\left(  \left \vert X\right \vert >x\right)  }{x\log x}%
=\infty.\label{critP}%
\end{equation}

Now, from (\ref{ltail}) and (\ref{rtail}) we get the following lemma, where,
as usual, the expression $f\left(  s\right)  \sim g\left(  s\right)  $ means
that $f\left(  s\right)  /g\left(  s\right)  $ tend to 1 when
$s\longrightarrow \infty.$

\begin{lemma}
\label{tstTW}(Two sided tails of the $\beta$-Tracy-Widom distribution) Let
$\beta>0,$ let $TW_{\beta}$ be a random variable $\beta$-Tracy-Widom
distributed. Then when $x\rightarrow \infty,$ it follows that
\[
P\left(  \left \vert TW_{\beta}\right \vert >x\right)  \sim P\left(  TW_{\beta
}>x\right)  =\exp \left(  -\tfrac{2}{3}\beta x^{3/2}\left[  1+o\left(
1\right)  \right]  \right)  .
\]

\end{lemma}

\begin{proof}
Using (\ref{ltail}) and (\ref{rtail}) we get
\begin{align*}
P\left(  \left \vert TW_{\beta}\right \vert >x\right)   & =P\left(  TW_{\beta
}<-x\right)  +1-P\left(  TW_{\beta}<x\right) \\
& =\exp \left(  -\frac{1}{24}\beta x^{3}\left(  1+o\left(  1\right)  \right)
\right)  +\exp \left(  -\tfrac{2}{3}\beta x^{3/2}\left[  1+o\left(  1\right)
\right]  \right)  ,
\end{align*}
and hence%
\begin{align*}
\lim_{x\rightarrow \infty}\frac{P\left(  \left \vert TW_{\beta}\right \vert
>x\right)  }{\exp \left(  -\tfrac{2}{3}\beta x^{3/2}\left[  1+o\left(
1\right)  \right]  \right)  }  & =1+\lim_{x\rightarrow \infty}\exp \left(
-\frac{1}{24}\beta x^{3}\left(  1+o\left(  1\right)  \right)  +\tfrac{2}%
{3}\beta x^{3/2}\left[  1+o\left(  1\right)  \right]  \right) \\
& =1.
\end{align*}

\end{proof}

Finally, using Lemma \ref{tstTW} it is possible to conclude that the $\beta
$-Tracy-Widom distribution is not infinite divisibility:

\begin{theorem}
\label{TWnoID}For any $\beta>0$ the $\beta$-Tracy Widom distribution is not
infinitely divisible.
\end{theorem}

\begin{proof}
Let assume that $X$ is infinitely divisible and given that neither it is
normal nor degenerate we must have that (\ref{critP}) is false, that is%
\[
\lim_{x\rightarrow \infty}\frac{-\log P\left(  \left \vert X\right \vert
>x\right)  }{x\log x}<\infty.
\]
However,
\begin{align*}
\lim \limits_{x\rightarrow \infty}\frac{-\log P\left(  \left \vert X\right \vert
>x\right)  }{x\log x}  & =\lim \limits_{x\rightarrow \infty}\frac{-1}{x\log
x}\log \left(  \tfrac{P\left(  \left \vert X\right \vert >x\right)  \exp \left(
-\tfrac{2}{3}\beta x^{3/2}\left[  1+o\left(  1\right)  \right]  \right)
}{\exp \left(  -\tfrac{2}{3}\beta x^{3/2}\left[  1+o\left(  1\right)  \right]
\right)  }\right) \\
& =\lim \limits_{x\rightarrow \infty}\frac{1}{x\log x}\left \{  -\log \left(
\tfrac{P\left(  \left \vert X\right \vert >x\right)  }{\exp \left(  -\tfrac{2}%
{3}\beta x^{3/2}\left[  1+o\left(  1\right)  \right]  \right)  }\right)
+\tfrac{2}{3}\beta x^{3/2}\left[  1+o\left(  1\right)  \right]  \right \} \\
& =\lim \limits_{x\rightarrow \infty}\frac{\tfrac{2}{3}\beta \sqrt{x}}{\log
x}\left[  1+o\left(  1\right)  \right] \\
& =\infty,
\end{align*}
the third equality follows from Lemma \ref{tstTW}.
\end{proof}

\begin{remark}
Taking $\beta=1,2$ or $4$ in Theorem \ref{TWnoID} the non infinite
divisibility of the classical Tracy-Widom distribution is deduced.
\end{remark}

\section{Non infinite divisibility in the finite $N$ case}

The following Lemma is necessary to determine the non infinite divisibility of
the largest eigenvalue of a random matrix of a GOE/GUE ensemble.

\begin{lemma}
\label{crit2}If $X$ is a non Gaussian real random variable such that
\[
P\left(  \left \vert X\right \vert >x\right)  \leq ae^{-bx^{c}}\text{
with}\ a,b>0,\ c>1,
\]
then $X$ is not infinitely divisible.
\end{lemma}

\begin{proof}
As in the proof of Theorem \ref{TWnoID} it is only necessary to prove that $X
$ satisfy the limit (\ref{critP}). Indeed,
\[
\lim \limits_{x\rightarrow \infty}\frac{-\log P\left(  \left \vert X\right \vert
>x\right)  }{x\log x}\geq \lim \limits_{x\rightarrow \infty}\frac{-\log \left(
ae^{-bx^{c}}\right)  }{x\log x}=\lim \limits_{x\rightarrow \infty}\frac{-\log
a+bx^{c}}{x\log x}=\infty.
\]

\end{proof}

Consider, $\lambda_{\max}^{N}$ the largest eigenvalue of a GOE ensemble of
dimension $N.$ In \cite[Lemma 6.3]{adg01} is proved that the two sided tails
of the largest eigenvalue of a GOE satisfy the following inequality%
\[
P\left(  \left \vert \lambda_{\max}^{N}\right \vert \geq x\right)  \leq
e^{-Nx^{2}/9}%
\]
and if $\lambda_{\max}^{N}$ the largest eigenvalue of a GUE ensemble of
dimension $N.$ The following inequality is deduced in \cite{l03}
\[
P\left(  \left \vert \lambda_{\max}^{N}-E\left(  \lambda_{\max}^{N}\right)
\right \vert \geq x\right)  \leq2e^{-2Nx^{2}}.
\]
from which, with help of Lemma \ref{crit2}, the next theorem follows:

\begin{theorem}
Let $\lambda_{\max}^{N}$ be the largest eigenvalue of a random matrix of a
GOE/GUE ensemble of random matrices. For all $N\geq2,\  \lambda_{\max}^{N}$ is
not infinitely divisible.
\end{theorem}

\begin{remark}
The case $N=2$ follows from Dom\'{\i}nguez-Molina and Rocha-Arteaga
\cite{dmra07}.
\end{remark}

\section{Discussion}

Recall that an infinitely divisible random variable, $X,$ in $\mathbb{R}_{+}$
must comply that $-\log P\left(  X>x\right)  \leq ax\log x,$ for some $a>0$
and $x$ sufficiently large, \cite{s79}. With this result it is possible to
deduce the non infinite divisibility of the following random
variables:\bigskip

I) Wigner surmise: $P\left(  s\right)  =\frac{\pi}{2}s\exp \left(  -\frac{\pi
}{4}s^{2}\right)  .$

II) The absolute value of $TW_{\beta},$ $Y_{\beta}=\left \vert TW_{\beta
}\right \vert $.

III) The truncation to the left or to the right of $TW_{\beta}$.\bigskip

\emph{Open problems}

\begin{enumerate}
\item \emph{Free infinite divisibility of the classical Tracy-Widom
distribution or the general} $\beta$-\emph{Tracy-Widom distribution.}

\item \emph{For each }$N\geq2,$ \emph{the non infinite divisibility of the
largest eigenvalue of random matrix of a GSE ensemble.}

\item \emph{Determine if the Tracy-Widom distribution is indecomposable.}

\item \emph{Investigate the infinite divisibility of }$\lambda_{\max}\left(
H_{N}^{\beta}\right)  $\emph{\ in the tridiagonal }$\beta$\emph{-Hermite
ensemble (\ref{Tridi}).} The article \cite{lr10} it may be useful.

\item \emph{Look for an infinitely divisible interpolation between Tracy-Widom
distribution and other distribution (except in the Tracy Widom case).
}Johanson \textbf{\cite{jo07} }discuss interpolation between the Gumbel
distribution, and the Tracy-Widom distribution. Bohigas et al
\textbf{\cite{bcp09}} discuss a continuous transition from Tracy-Widom
distribution to the Weilbull distribution, and from Tracy-Widom distribution
to Gaussian distribution. It is not known if these interpolations are
infinitely divisible (except in the Tracy Widom case).
\end{enumerate}

\end{document}